\documentclass[a4paper, twoside,12pt]{article}
\usepackage{fancyhdr}
\usepackage{fnpos}
 \usepackage[english]{babel}        
\usepackage[T1]{fontenc}          
\usepackage{graphicx}             
\usepackage{makeidx}
\usepackage{fancybox}
\usepackage{framed}
\usepackage{fancyhdr}
 \usepackage{pstricks,pst-plot,pstricks-add}
\usepackage[margin=1in]{geometry}
\usepackage{graphicx}
\usepackage{titlesec}
\usepackage{amsmath}
\usepackage{amsfonts}   
\usepackage{amssymb}    
\usepackage{amsthm}
\usepackage{dsfont}
\usepackage{mathtools}
\usepackage{verbatim}   
\usepackage{color}
\usepackage{listings}
\usepackage{graphicx}
\usepackage{amsmath}
\usepackage{amsmath,amssymb,color,inputenc,euscript,graphicx,psfrag}

\newtheorem{thm}{Theorem}[section]
\newtheorem{cor}[thm]{Corollary}
\newtheorem{lem}[thm]{Lemma}
\newtheorem{prop}[thm]{Proposition}
\newtheorem{defn}[thm]{Definition}

\numberwithin{equation}{section}

\renewcommand{\thefootnote}

\newcommand\co{\operatorname{co}}

\renewcommand\Re{\operatorname{Re}}

 {  }{  }

\author { B\'echir Amri$^*$ and Amel Hammi$^{**}$  }

\title{   Semigroup and Riesz transform   for the Dunkl- Schr\"{o}dinger operators }

\date{ }

\begin{document}
 \maketitle
\begin{center}
$^*$Taibah University, College of Sciences, Department of Mathematics, P. O. BOX 30002, Al Madinah AL Munawarah, Saudi Arabia.\\
\textbf{ e-mail:} bechiramri69@gmail.com\\
  $^{**}$Universit\'{e} Tunis El Manar, Facult\'{e} des sciences de Tunis,\\ Laboratoire d'Analyse Math\'{e}matique
       et Applications,\\ LR11ES11, 2092 El Manar I, Tunisie.\\
     \textbf{ e-mail:}
   hammiamel097@gmail.com
\end{center}
 \begin{abstract}
 Let  $L_k=-\Delta_k+V$ be the Dunk- Schr\"{o}dinger operators, where  $\Delta_k=\sum_{j=1}^dT_j^2$ is the Dunkl Laplace operator associated to the dunkl operators $T_j$  on $\mathbb{R}^d$ and   $V$ is a nonnegative potential function.
  In the first part of  this paper  we introduce the  Riesz transform   $R_j= T_j L_k^{-1/2}$  as an $L^2$- bounded operator  and we prove  that is
   of weak type $(1,1)$ and then is  bounded  on $L^p(\mathbb{R}^d,d\mu_k(x))$ for $1<p\leq 2$. The second pat is devoted to
  the $L^p$ smoothing of the semigroup generated by   $L_k$, when $V$ belongs to the standard  Koto class.
 \end{abstract}
\section{Introduction and backgrounds   }
Nowadays the Dunkl  analysis on $\mathbb{R}^d$    becomes more and more extensive  to include topics  in the  classical Fourier analysis.
The introduction of the Dunkl operators and Dunkl-Fourier transform are  fundamental tools for generalization of a  classical known  results.
In Dunkl analysis the  doubling condition on the underlying measure is   satisfied \cite{B3,AJD} and so, it is a  convenient setting for
developing Calder\`{o}n-Zygmund Theory \cite{B1,B2,B3,JD1}.
  In this paper we consider Schr\"{o}dinger operators $L_k=-\Delta_k+V$ associated to the Dunkl Laplace operator on $\mathbb{R}^d$ given by
$\Delta_k=\sum_{j=1}^{^d}T_j^2$
 where $T_j$ are  a family of differential-difference operators associated to a finite reflection group, which are called Dunkl
operators.
 When $V$ is a nonnegative locally integrable  function the operator $L_k$ is essentially self-adjoint and its closure generates a semigroup of self-adjoint linear contractions $W_t= e^{-tL_k}$, $t>0$ (see \cite[Th 3.1]{BH} ). The  kernels $W_t(x,y) $ of this semigroup   possess the  Gaussian upper bounds (  see for instances \cite{AJD, JD1, JD3}) which are very useful ingredients  and  actually was the main motivation for    the study  of the $L^p$- boundedness of  the generalized  Riesz transform  $R_j=  T_j L_k^{-1/2}$, $j=1,...,d$ and the $L^p-L^q$ regularities of $W_t$, $t>0$.
 First, we establish
    that    $R_j$  is  of weak type $(1,1)$ and then is  bounded  on $L^p(\mathbb{R}^d,w_k(x))$ for $1<p\leq 2$. The method  which  inspired in part by the ideas in \cite{CD} consists in applying the usual  Calder\'{o}n-Zygmund decomposition and  a weighted $L^2$- estimates for
  the kernel $T_jW_t(.,y)$. ). In   \cite{Sim1} the smoothing problem on $L^p$ for the classical  Schrodinger semigroups   was studied  using  analytical method. We will investigate this technique to  prove that  if $V$ is a non-negative potential, belongs to the standard  Koto class then    $W_t$  maps $L^p(\mathbb{R}^d,d\mu_k)$    into  $L^q(\mathbb{R}^d,d\mu_k)$ for all $1\leq p\leq q\leq \infty$.

 \par We start by recalling some backgrounds from Dunkl's analysis. For details, we refer to \cite{J1,D1, D2, R2,XT} and the references cited there.
\par Consider the Euclidian space  $\mathbb{R}^d$  equipped with the canonical basis  $(e_1,e_2,...,e_d)$ and the  scalar product $\langle x,y\rangle=\sum_{j=1}^dx_jy_j$
 with associated   norm $|x|=\langle x, x\rangle^{1/2}$.
   Let $G\!\subset\!\text{O}(\mathbb{R}^d)$
be a finite reflection group associated to a reduced root system $R$
and $k:R\rightarrow[0,+\infty)$ be  a $G$--invariant function
(called multiplicity function).
Let $R^+$ be a positive root subsystem. The Dunkl operators \,$T_\xi$ on $\mathbb{R}^d$ are
the following $k$--de\-for\-ma\-tions of directional derivatives $\partial_\xi$
by difference operators\,:
 \begin{equation}\label{DO}
T_\xi f(x)=\partial_\xi f(x)
+\sum_{\,\alpha\in R^+}\!k(\alpha)\,\langle\alpha,\xi\rangle\,
\frac{f(x)-f(\sigma_\alpha.\,x)}{\langle\alpha,\,x\rangle}\,,
 \end{equation}
where  $\sigma_\alpha $ denotes the reflection
with respect to the hyperplane orthogonal to $\alpha$. It's given by
\begin{equation}\label{sig}
  \sigma_\alpha(x)=x-\langle x,\alpha\rangle\alpha.
\end{equation}
Here we assume further that $|\alpha|^2=2$ for all $\alpha\in R$.
\par The operators $\partial_\xi$ and $T_\xi$
are intertwined by a Laplace--type operator
\begin{eqnarray*}\label{vk}
V_k\hspace{-.25mm}f(x)\,
=\int_{\mathbb{R}^d}\hspace{-1mm}f(y)\,d\nu_x(y),
\end{eqnarray*}
associated to a family of compactly supported probability measures
\,$\{\,\nu_x\,|\,x\!\in\!\mathbb{R}^d\hspace{.25mm}\}$\,.
Specifically, \,$\nu_x$ is supported in the the convex hull $\co(G.x)$
  and  satisfies
 \begin{equation}\label{nu}
    \nu_{rx}(B)=\nu_x(r^{-1}B), \qquad \nu_{g.x}(B)=\nu_x(g^{-1}.B)
 \end{equation}
  for each $r > 0$, $g \in G$ and each Borel set $ B \subset \mathbb{R}^d$.

\par The Dunkl operators are antisymmetric
with respect to the measure $d\mu_k(x)=w_k(x)\,dx$
with
$$
w_k(x)=\,\prod_{\,\alpha\in R^+}|\,\langle\alpha,x\rangle\,|^{\,2\,k(\alpha)}\,.
$$
It has been shown in \cite{AJD} that  there exist $c,C>0$ such that
 for $x\in \mathbb{R}^d$ and $r>0$,
 \begin{equation}\label{888}
 c r^d\prod_{\alpha\in R}(|\langle x,\alpha\rangle|+r)^{k(\alpha)}\leq \mu_k(B(x,r))\leq  C r^d\prod_{\alpha\in R}(|\langle x,\alpha\rangle|+r)^{k(\alpha)}
 \end{equation}

and for  $0<r< R$,
  \begin{equation}\label{Rr}
     C^{-1}\left(\frac{R}{r}\right)^{d+2\gamma_k}\leq\frac{\mu_k(B(x,R))}{\mu_k(B(x,r))}\leq C\left(\frac{R}{r}\right)^{d+2\gamma_k}
\end{equation}
where  $B(x,r)$ denotes the  Euclidean ball centered at $x$ with radius $ r$ and $\gamma_k$ is given by
$$\gamma_k=\sum_{\alpha\in R^+}k(\alpha).$$
 This implies that
      $\mu_k$ is a doubling measure that is for some constant $C>0$
  $$\mu_k(B(x,2r))\leq C\; \mu_k(B(x,r)) ,\quad \text{ for} \; x\in \mathbb{R}^d\;  \text{and}\; r>0.$$
     \par For every $y\!\in\!\mathbb{C}^d$\!,
the simultaneous eigenfunction problem
\begin{equation*}
T_\xi f=\langle y,\xi\rangle\,f
\qquad\forall\;\xi\!\in\!\mathbb{R}^d
\end{equation*}
has a unique solution $f(x)\!=\!E_k(x,y)$
such that $E_k(0,y)\!=\!1$, called the Dunkl kernel and is given by
\begin{equation}\label{EV}
E_k(x,y)\,
=\,V(e^{\,\langle\;.,y\;\rangle})(x)\,
=\int_{\mathbb{R}^d}\hspace{-1mm}e^{\,\langle \eta ,y\rangle}\,d\nu_x(\eta )
\qquad\forall\;x\!\in\!\mathbb{R}^d.
\end{equation}
Furthermore this kernel has a holomorphic extension to $\mathbb{C}^d\times \mathbb{C}^d $
 and the following   hold\,: for  $ \;x, \;y\!\in\!\mathbb{C}^d,$
\begin{itemize}
\item[(ii)] $E_k(x,y)=E_k(y,x)$,
\item[(iii)] $E_k(\lambda x,y)=E_k(x,\lambda  y)$, for $\lambda\in \mathbb{C}$
\item[(iv)] $E_k(g. x,g.y)=E_k(x, y)$, for $g\in G$.
\end{itemize}
\par  The Dunkl transform  is defined on $L^1(\mathbb{R}^d\!,d\mu_k)$ by
$$
\mathcal{F}_kf(\xi)= \frac{1}{c_k}\;
\int_{\mathbb{R}^d}\!f(x)\,E_k(x,-i\,\xi) d\mu_k(x)
$$
where
$$
c_k\,=\int_{\mathbb{R}^d}\!e^{-\frac{|x|^2}2} d\mu_k(x).
$$
 Dunkl transform is a generalization of the Fourier transform $(k=0)$
and  satisfies the following properties:
\begin{itemize}
\item[(i)]
$\mathcal{F}_k$ is a topological automorphism
of  $\mathcal{S}(\mathbb{R}^d)$,  the Schwartz space of rapidly decreasing
functions on $\mathbb{R}^d$.
\item[(ii)]
(\textit{Plancherel Theorem\/})
$\mathcal{F}_k$  extends to
an isometric automorphism of $L^2(\mathbb{R}^d\!,d\mu_k)$.
\item[(iii)] (Parseval's formula). For all $f, g \in L^2(\mathbb{R}^d,d\mu_k)$
 we have
$$\int_{\mathbb{R}^d}f(x)\overline{g(x)}d\mu_k(x)=\int_{\mathbb{R}^d}\mathcal{F}_k(f)(x)\overline{\mathcal{F}_k(g) (x)}d\mu_k(x)$$
\item[(vi)]
(\textit{Inversion formula\/})
For every $f\!\in\!\mathcal{S}(\mathbb{R}^d)$,
and more generally for every $f\!\in\!L^1(\mathbb{R}^d\!,d\mu_k)$
such that $\mathcal{F}_kf\!\in\!L^1(\mathbb{R}^d\!, d\mu_k)$,
we have
$$
f(x)=\mathcal{F}_k^2\!f(-x)\qquad\forall\;x\!\in\!\mathbb{R}^d.
$$
\end{itemize}
\par Let $x\in \mathbb{R}^d$, the
Dunkl translation operator $\tau_x$ is given for $f\in
L^2_k(\mathbb{R}^d,d\mu_k)$ by
\begin{eqnarray*}\label{dutr}
\mathcal{F}_k(\tau_x(f))(y)= \mathcal{F}_kf(y)\,E_k(x,iy), \quad
y\in\mathbb{R}^d.
\end{eqnarray*}
 In the case when $f(x)=\widetilde{f}(|x|)$ is a radial function in  $  \mathcal{S}(\mathbb{R}^d)$,  the Dunkl translation is represented by the following integral
\begin{eqnarray}\label{trad}
\tau_x(f)(y)=
\int_{\mathbb{R}^{n}}\widetilde{f}( \sqrt{|y|^2+|x|^2+2<y,\eta>}\;)\;d\nu_x(\eta).
 \end{eqnarray}

\par We define  the Dunkl convolution product   for suitable functions $f$ and $g$ by
$$f*_kg(x)=\int_{\mathbb{R}^d} \tau_x(f)(-y)g(y)d\mu_k(y),\quad x\in\mathbb{R}^d.$$
We note that it is commutative and satisfies,
\begin{eqnarray*}\label{conv}
 \mathcal{F}_k(f*_kg)=\mathcal{F}_k(f)\mathcal{F}_k(g), \quad \quad f,\;g\in L^2(\mathbb{R}^d,d\mu_k).
\end{eqnarray*}

\par  In  what follows  present our main tools  for Dunkl Scr\"{o}dinger  semigroup. We refer to \cite{BH} for further details of some facts.
  \par The Dunkl Laplacian operator is given  by
 $$\Delta_k=\sum_{j=1}^dT_{j}^2,$$
 where $T_j=T_{e_j}$.  We consider   $-\Delta_k$ as a densely defined   operator on the Hilbert space $L^2(\mathbb{R}^d,d\mu_k)$  which is symmetric  and positive. We denote by $A_k$   the unique positive self-adjoint  extension   of $-\Delta_k$,  defined by
  $$D(A_k)=H_k^2(\mathbb{R}^d)=\{f\in L^2(\mathbb{R}^d,d\mu_k);\; |x|^2\mathcal{F}_k(f)\in L^2(\mathbb{R}^d,d\mu_k)\}$$
  $$ A_k(f) =\mathcal{F}_k^{-1}(|x|^2\mathcal{F}_k(f)); \qquad f \in D(A_k).$$
The operator $A_k$ is generator of a strongly continuous one parameter semi group $(e^{-tA_k})_{t\geq 0}$
where
$$e^{-tA_k}f=  \mathcal{F}_k^{-1} (e^{-t|.|^2} \mathcal{F}_k(f)), \qquad
 f\in L^2(\mathbb{R}^d,d\mu_k).$$
  It follows that  $e^{-tA_k}$ is an integral operator given by
\begin{equation}\label{1}
   e^{-tA_k}f(x)==k_t*_kf =\int_{\mathbb{R}^d}K_t(x,y)f(y)w_k(y)dy
\end{equation}
with
\begin{equation*}\label{kk}
 k_t(x)= \mathcal{F}_k^{-1}(e^{-t|.|^2})(x) = \frac{1}{c_kt^{\gamma_k+d/2}}\;e^{-|x|^2/4t}
\end{equation*}
and
\begin{equation*}\label{KK}
 K_t(x,y)=\tau_x(k_t)(-y)= \frac{1}{c_kt^{\gamma_k+d/2}}\;e^{- (|x|^2+|y|^2)/4t}E_k(x/2t,y).
\end{equation*}
We call it the Dunkl heat kernel. We have a bound on  $K_t$ of the form (see \cite{AJD})
\begin{equation}\label{dhk}
  0<K_t(x,y)\leq C\frac{e^{-c |x^+-y^+|^2/t }}{\max(\mu_k(B(x, \sqrt{t})),\mu_k(B(y,\sqrt{t})))}
  \end{equation}
  where $x^+$
  is the unique element of $G.x$ that contained in the closed   fundamental Weyl chamber $C^+=\{x\in \mathbb{R}^d;\; \langle x,\alpha\rangle\geq 0\}$.
 Noting here that
\begin{equation}\label{12}
  |x^+-y^+|= \min_{g\in G}|g.x-y|.
\end{equation}
From  (\ref{888}),   the heat kernel  $K_t$ satisfies
\begin{equation}\label{dhhk}
  K_t(x,y)\leq C t^{-d/2}   \leq  \frac{e^{-c  |x^+-y^+|^2/t}}{ \max( w_k(x),w_k(y))}
  \end{equation}
  and
\begin{equation}\label{dhhhk}
  K_t(x,y)\leq C t^{-d/2-\gamma_k}\;e^{-c  |x^+-y^+|^2/t}
  \end{equation}
 \par Let $V$ be a nonnegative measurable function on $\mathbb{R}^d$ that is finite almost everywhere.
 In   the Hilbert space $L^2(\mathbb{R}^d,d\mu_k)$ we consider the  operator
 $$\mathcal{L}_k=A_k+V$$
  with domain   $D(\mathcal{L}_k)=H_k^2(\mathbb{R}^d)\cap D(V)$ where
 $$D(V)=\{ f\in L^2(\mathbb{R}^d,d\mu_k);\; Vf\in L^2(\mathbb{R}^d,d\mu_k)\}.$$
 We call this operator the Dunkl Schr\"{o}dinger operator.   Define the quadratic form $q_k$ by
 \begin{eqnarray*}
 D(q_k)&=&\{f\in L^2(\mathbb{R}^d,d\mu_k);\;\left(\sum_{j=1}^n|T_jf|^2\right)^{1/2},   V^{1/2}f \in  L^2(\mathbb{R}^d,d\mu_k) \}\\
 q_k(f)&=&\sum_{j=1}^n\|T_jf\|_{2,k}^2+\|V^{1/2}f\|^2_{2,k}.
\end{eqnarray*}
The quadratic form  $q_k$    is densely defined   and  closed, then   there exists  a unique positive self adjoint operator $L_k$ such that:
    $$q_k(\varphi)=\langle L_k(\varphi),\varphi \rangle,\qquad \varphi \in D(q_k),$$
    Moreover,
 \begin{equation}\label{l2}
 D(q_k)=D(L_k^{1/2})\quad\text{and}\quad q_k(\varphi)=\|L_k^{1/2}(\varphi)\|_{2,k}.
\end{equation}
Noting here that $L_k^{1/2}$ is the  unique positive self-adjoint operator  such that
$(L_k^{1/2})^2=L_k$.
 \par When  $V\in L^2_{loc}(\mathbb{R}^d,d\mu_k) $ and $V\geq0$, we have the following:
  \begin{itemize}
    \item[(i)]    $\mathcal{L}_k$ is essentially self-adjoint on $C^\infty_0(\mathbb{R}^d)$
and its closure  is $L_k$.
     \item[(ii)] For $u\in L^2(\mathbb{R}^d,d\mu_k)$  we have

    \begin{equation}\label{LA}
   | e^{-tL_k}(u)|\leq e^{-t\mathcal{A}_k}(|u|).
  \end{equation}
\item[(iii)]  $ e^{-tL_k}$, $t>0$ is an integral operator,
  and its   kernel $W_t$ satisfies
\begin{equation}\label{pettis}
  0\leq W_t(x,y) \leq  K_t(x,y).\end{equation}
  \end{itemize}

  \section{Riesz transforms of Dunkl-Schr\"{o}dinger operators}
 In this section we assume that $V\geq 0$ and $V\in L^2_{loc}(\mathbb{R}^d,d\mu_k)$. Let us begin by considering the following observation.
 \par From (\ref{l2})  we see  that for $j=1,...,d$
 \begin{equation}\label{l3}
   \|\xi_j\mathcal{F}_k(\varphi)\|_{2,k}=\|T_j\varphi\|_{2,k}\leq \|q_k(\varphi)\|_{2,k}=\|L_k^{1/2}\varphi\|_{2,k}; \quad \varphi \in D(L_k^{1/2})
 \end{equation}
 which implies  that $ker(L_k^{1/2}) = \{0\}$. But since $Ran(L_k^{1/2})^{\perp}=ker(L_k^{1/2})$ then one has
$$\overline{Ran(L_k^{1/2})}= L^2(\mathbb{R}^d,d\mu_k).$$
Let $L_k^{-1/2}$ the inverse operator of $L_k^{1/2}$ defined on $Ran(L_k^{1/2})$. It follows from  (\ref{l3}) that
$$\|T_jL_k^{-1/2}u\|_{2,k}\leq \|u\|_{2,k}; \quad u \in Ran(L_k^{1/2})$$
 Thus   $T_jL_k^{-1/2}$, $j=1,...,d$, can be extended ( by density)  to a bounded operator on  $L^2(\mathbb{R}^d,d\mu_k)$, we call it  Riesz transform of
  the Dunkl- schrodinger operator $L_k$ and   will be  denoted by  $R_j$.
  \par Now  from functional calculus  for positive self-adjoint operator  one can write
  \begin{equation}\label{1/2}
    L_k^{-1/2}u=\frac{1}{\sqrt{\pi}}\int_0^\infty e^{-sL_k}u \; \frac{ds}{\sqrt{s}}.
\end{equation}
 and  the   Riesz transform $R_j$ can represented  by
 $$R_j=T_jL_k^{-1/2}=\frac{1}{\sqrt{\pi}}\int_0^\infty T_j\;e^{-tL_k}\frac{dt}{\sqrt{t}}.$$
 Our main result is the following.
 \begin{thm}\label{999}
  The Riesz transform  $R_j$, j=1,...,d, is
of weak type $(1, 1)$ and then  is bounded on  $L^p(\mathbb{R}^d,d\mu_k)$, $1<p\leq 2$.
\end{thm}
In the next,  we state an  estimates which are crucial for the  the proof of Theorem   \ref{999}.
Let $\phi$ be the   function
 $$\phi(x,y)= |G|^{-1}\sum_{g\in G} \int_{\mathbb{R}^d}e^{\sqrt{1+  A(x,y,\eta)^2}}\;d\nu_{g.y}(\eta)=\int_{\mathbb{R}^d}e^{\sqrt{1+  A(x,y,\eta)^2}}\;d\nu_y^G(\eta)$$
 where $d\nu_y^G$ denotes the $G-invariant$ measure,
$$d\nu_y^G= |G|^{-1}\sum_{g\in G} d\nu_{g.y}(\eta).$$
( $d\nu$  is the measure given by the formulas $(\ref{vk})$ ) and
$$ A(x,y,\eta)=\sqrt{|y|^2+|x|^2-2<x,\eta>}= \sqrt{|x-\eta|^2+|y|^2-|\eta|^2};\quad \eta \in conv(G.y).$$
 Notice that   $\phi(.,y)$ is $C^\infty$ function which is $G-invariant$.
 \begin{thm}\label{101}
 The kernel $W_t(x,y)$ satisfies the following estimates
 \begin{eqnarray}
 \int_{\mathbb{R}^d}|T_jW_t(x,y)|^2\phi(x/\sqrt{t},y/\sqrt{t})d\mu_k (x)&\leq& Ct^{-\gamma_k-d/2-1}, \label{01}
\\\int_{ |x^+-y^+|>\sqrt{t}}|T_jW_s(x,y)| d\mu_k(x) &\leq& \frac{C}{\sqrt{s}} \;e^{-c\sqrt{t/s}}  \quad j=1,...,d,\label{02}
 \end{eqnarray}
   where the constant $C > 0$ is independent of $t,s$ and $y\in \mathbb{R}^d $.
\end{thm}
 Let us first  observe   since $L_ke^{-tL_k}$ is bounded operator on $L^2(\mathbb{R}^d,d\mu_k)$ and
 $W_t(x,y)=e^{-(t/2)L_k}(W_{t/2}(.,y))(x)$ then $W_t(.,y)$  belong to the domain of the operator $L_k$ and  $T_jW_t(x, y)$ is well defined  in
 $L^2(\mathbb{R}^d,d\mu_k)$.
\par We will give  the proof of the theorem \ref{101} after proving Theorem \ref{999}.
\begin{proof}[Proof of  Theorem \ref{999}.]
     We start  with  the  Calder\'{o}n-Zygmund decomposition.
  For $f\in L^1(\mathbb{R}^d,d\mu_k)\cap L^2(\mathbb{R}^d,d\mu_k)$ and $\lambda>0$, there is a decomposition
  $f=g+\sum_i b_i =g+b$ such that the following   hold:\\
  (i) $|g(x)|\leq C\;\lambda$, for a.e. $x\in \mathbb{R}^d$.
 \\(ii)  there exists a sequence of balls $B(x_i; r_i)$  such that the support of
each $b_i$ is contained in $B_i=B(x_i; r_i)$ and
$\|b_i\|_{1,k}\leq C\;\lambda \mu_k(B_i).$
\\(iii)  $\displaystyle{\sum_i\mu_k(B_i)\leq C\;\frac{\|f\|_{1,k}}{\lambda}}$
 \\(iv) each  $x\in \mathbb{R}^d$ is contained in at most a finite number   of the balls  $B_i$.
 \par The proof of Theorem \ref{999}   consists of showing   the following    inequality
   \begin{equation}\label{11}
 \mu_k (\{x; |R_j f(x)| > \lambda\})\leq C\frac{\|f\|_{1,k}}{\lambda}.
  \end{equation}
In view of  the  Calder\'{o}n-Zygmund decomposition of $f$
$$\mu_k (\{x; |R_j (f)(x)| > \lambda\}) \leq
 \mu_k (\{x; |R_j g(x)| > \lambda/2\}) + \mu_k(\{x; |R_j (b)(x)| > \lambda/2\}).$$
Since  $R_j$ is bounded on $L^2(\mathbb{R}^d,d\mu_k)$  then we get by using  (i),
$$ \mu_k (\{x; |R_j g(x)| > \lambda/2\}) \leq C \;\frac{\|R_j(g)\|_{2,k}^2}{\lambda^2}\leq   C \;\frac{\|g\|_{2,k}^2}{\lambda^2}\leq C\; \frac{\|f\|_{1,k} }{\lambda}$$
Next we write
 $$R_j(b_i)= R_je^{-t_iL_k}b_i+ R_j(I-e^{-t_iL_k})b_i$$
where $t_i = r_i^2$  ($r_i$ is the radius of $B_i$). We claim that
$$ \mu_k(\{x; |\sum_iR_je^{-t_iL_k}b_i| > \lambda/4\})+\mu_k(\{x; |\sum_i (R_j(I-e^{-t_iL_k})b_i)| > \lambda/4\})\leq
 C\frac{\|f\|_{1,k}}{\lambda}. $$
Using  (\ref{LA}), (\ref{dhk}) and (ii) we have that

\begin{eqnarray*}
   |e^{-t_iL_k}b_i(x)|  &\leq  & C\; \int_{B_i} \frac{e^{-c\frac{|x^+-y^+|^2}{t_i}}}{ \mu_k(B(x, \sqrt{t_i})) )}|b_i(y)|
   d\mu_k(y)
  \\ &\leq & C\;\sum_{g\in G}\int_{B_i} \frac{e^{-c\frac{|g.x-y|^2}{t_i}}}{ \mu_k(B(g.x, \sqrt{t_i})) )}|b_i(y)|
 d\mu_k(y)
  \\ &\leq & C\;\sum_{g\in G} \frac{e^{-c'\frac{|g.x-x_i|^2}{t_i}}}{ \mu_k(B(g.x, \sqrt{t_i})) )}\int_{B_i}|b_i(y)|
 d\mu_k(y)
  \\ &\leq &C\;\lambda\; \sum_{g\in G} \frac{e^{-c'\frac{|g.x-x_i|^2}{t_i}}}{ \mu_k(B(g.x, \sqrt{t_i})) )}  \mu_k(B_i)
  \\ &\leq & C\;\lambda\;\sum_{g\in G}\frac{1}{ \mu_k(B(g.x, \sqrt{t_i})) )}\int_{ \mathbb{R}^d}  e^{-c''\frac{|g.x-y|^2}{t_i}} \mathds{1}_{B_i}(y) d\mu_k(y).
\end{eqnarray*}
Here $C,c,c',c''$ are positive
constants. Now  we note that
\begin{eqnarray*}
 &&\left\| \sum_ie^{-t_iL_k}b_i\right\|_{2,k}^2\leq C\lambda\left\|\sum_{i}\sum_{g\in G} \frac{1}{ \mu_k(B(g. , \sqrt{t_i})) )}\int_{ \mathbb{R}^d}  e^{-c''\frac{|g. -y|^2}{t_i}} \mathds{1}_{B_i}(y)d\mu_k(y)\right\|_{2,k}\\
  &&\leq C\lambda\;\sup_{\|u\|_{2,k}=1} \sum_{i}\sum_{g\in G}\int_{\mathbb{R}^d}\int_{ \mathbb{R}^d}\frac{e^{-c''\frac{|gx -y|^2}{t_i}}}{ \mu_k(B(gx , \sqrt{t_i})) )}| u(x)| \mathds{1}_{B_i}(y)d\mu_k(y)d\mu_k(x).
  \end{eqnarray*}
   Since $B(y , \sqrt{t_i})\subset B(g.x , |g.x-y|+\sqrt{t_i}) $ then from (\ref{Rr}) we have that
  $$\mu_k(B(y , \sqrt{t_i}))  \leq \left( 1+\frac{|g.x-y|}{\sqrt{t_i}}\right)^{d+2\gamma_k}\mu_k(B(g.x , \sqrt{t_i}))$$
Putting $c_1=c''/2$ and using the fact that
$ s\rightarrow (1+s)^{d+2\gamma_k}e^{-c_1s^2}$ is bounded it follows that
\begin{eqnarray*}
&&\sum_{i}\sum_{g\in G}\int_{ \mathbb{R}^d}\frac{e^{-c''\frac{|gx -y|^2}{t_i}}}{ \mu_k(B(gx , \sqrt{t_i})) )}| u(x)|  d\mu_k(x)
 \\&&\qquad\qquad\qquad\qquad\leq \sum_{i}\sum_{g\in G}\frac{1}{\mu_k(B(y , \sqrt{t_i}))}\int_{ \mathbb{R}^d}e^{-c_1\frac{|gx -y|^2}{t_i}}| u(x)|d\mu_k(x).
\end{eqnarray*}
Writing
\begin{eqnarray*}
 \int_{ \mathbb{R}^d}e^{-c_1\frac{|gx -y|^2}{t_i}}| u(x)|  d\mu_k(x)&=&  \int_{ |g.x-y|< \sqrt{t_i} }e^{-c_1\frac{|gx -y|^2}{t_i}}| u(x)| d\mu_k(x) \\ &+&  \sum_{p=0}^\infty \int_{ 2^p\sqrt{t_i}\leq |g.x-y|\leq 2^{p+1} \sqrt{t_i} }e^{-c_1\frac{|gx -y|^2}{t_i}}| u(x)| d\mu_k(x)
 \end{eqnarray*}
$$\leq \int_{ |x-g^{-1}.y|< \sqrt{t_i} } | u(x)| d\mu_k(x) +  \sum_{p=0}^\infty e^{-c_1 2^{2p}}\int_{    |x-g^{-1}.y|\leq 2^{p+1} \sqrt{t_i} } | u(x)|
 d\mu_k(x)$$
$$\leq \left( \mu_k(B(y , \sqrt{t_i}))+ \sum_{p=0}^\infty e^{-c_1 2^{2p}} \mu_k(B(y , 2^{p+1}\sqrt{t_i}))\right) Mu(g^{-1}.y) $$
where
$$Mu(y)=\sup_{r>0}\frac{1}{\mu_k(B(y,r)))}\int_{|x-y|\leq r}|u(x)|d\mu_k(x)$$
is the Hardy-Littlewood maximal function. Noting  here that $\mu_k$ is $G$-invariant. Making the use of (\ref{Rr}),
\begin{eqnarray*}
 \frac{1}{\mu_k(B(y , \sqrt{t_i}))}\int_{ \mathbb{R}^d}e^{-c_1\frac{|gx -y|^2}{t_i}}| u(x)|d\mu_k(x)& \leq&
\left(1+\sum_{p=0}^\infty 2^{(p+1)(d+2\gamma_k)}e^{-c_1 2^{2p}}\right)Mu(g^{-1}.y)
\\&\leq &C Mu(g^{-1}.y).
\end{eqnarray*}
Hence we obtain
\begin{eqnarray*}
&&\left\| \sum_{i}\sum_{g\in G}\frac{1}{ \mu_k(B(g. , \sqrt{t_i})) )}\int_{ \mathbb{R}^d}  e^{-c'\frac{|g. -y\|^2}{t_i}} \mathds{1}_{B_i}(y)d\mu_k(y)\right\|_{2,k}\qquad\qquad\qquad\qquad\qquad\qquad
\\&&\qquad\qquad\qquad\qquad\qquad\qquad\qquad\leq C\; \sum_{g\in G}\sup_{\|u\|_{2,k}=1}\int_{ \mathbb{R}^d} Mu(g^{-1}.y)\sum_{i}\mathds{1}_{B_i}(y)d\mu_k(y)
\\&&\qquad\qquad\qquad\qquad\qquad\qquad\qquad\leq \left\|\sum_i\mathds{1}_{B_i}\right\|_{2,k},
\end{eqnarray*}
since  the maximal function   $M$ is bounded on $L^2(\mathbb{R}^d,d\mu_k)$. Therefore,
$$\left\| \sum_ie^{-t_iL_k}b_i\right\|_{2,k}^2\leq C\lambda^2\left\|\sum_i\mathds{1}_{B_i}\right\|_{2,k}^2
\leq C\lambda^2\sum_i\mu_k(B_i)\leq C\lambda\|f\|_{1,k} $$
and from the $L^2$-booundedness of $R_j$ we have
$$\mu_k\left\{x;\;  R_j\left(\sum_ie^{-t_iL_k}b_i\right)(x)>\lambda/4\right\}\leq C\; \frac{\|f\|_{1,k}}{\lambda}.$$
\par Consider now the term $\displaystyle{R_j\left(\sum_i(I-e^{-t_iL_k}\right)b_i}$. Let
$$B_i^*=\bigcup_{g\in G}B(g.x_i,2r_i).$$
We write
\begin{eqnarray*}
&&\mu_k\left\{x;\;  R_j\left(\sum_i(I-e^{-t_iL_k})b_i\right)(x)>\lambda/4\right\}\leq
\\&& \qquad\qquad\qquad\sum_i\mu_k(B_i^*)+ \mu_k\left\{x\notin\bigcup_iB_i^*;\;  R_j\left(\sum_i(I-e^{-t_iL_k})b_i\right)(x)>\lambda/4\right\}
 \end{eqnarray*}
The doubling volume property of the measure $\mu_k$ and property  (iii) in the  Calderon-Zygmund  decomposition imply that
$$\sum_i\mu_k(B_i^*)\leq  C\; \frac{\|f\|_{1,k}}{\lambda}.$$
On the other hand we have
\begin{eqnarray}
\nonumber &&  \mu_k\left\{x\notin\bigcup_iB_i^*;\;  R_j\left(\sum_i(I-e^{-t_iL_k})b_i\right)(x)>\lambda/4\right\}\qquad\qquad\qquad
\\&&\qquad\qquad\qquad\qquad\qquad\qquad\qquad
\leq \frac{4}{\lambda} \sum_i\int_{ \mathbb{R}^d\setminus B_i^*} |R_j(I-e^{-t_iL_k})(b_i)(x)|\mu_k(x)|.\label{012}
\end{eqnarray}
Here  we use   the integral representation of the operator $R_j(I-e^{-t_iL_k})$ by a kernel  $\mathcal{K}_i$. In fact,
from (\ref{1/2}),
\begin{eqnarray*}
L_k^{-1/2}(I-e^{-t_iL_k})&=& \frac{1}{\sqrt{\pi}}\int_{0}^{+\infty}\frac{e^{-sL_k}}{\sqrt{s}}ds-\frac{1}{\sqrt{\pi}}\int_{0}^{+\infty}\frac{e^{-(s+t_i)L_k}}{\sqrt{s}}ds
\\&  =& \frac{1}{\sqrt{\pi}} \int_{0}^{+\infty}\left(\frac{1}{\sqrt{s}}- \frac{\mathds{1}_{\{s>t_i\}} }{\sqrt{s-t_i}}\right) e^{-sL_k}ds
\end{eqnarray*}
 Putting  for $ x\notin G.y $
$$\mathcal{K}_i(x,y)= \frac{1}{\sqrt{\pi}}\int_{0}^{+\infty}\left(\frac{1}{\sqrt{s}}- \frac{\mathds{1}_{\{s>t_i\}} }{\sqrt{s-t_i}}\right)T_jW_s(x,y)ds
= \int_{0}^{+\infty}g_i(s)T_jW_s(x,y)ds, $$
it yield   that
$$ R_j(I-e^{-t_iL_k})(u)(x)=T_j L_k^{-1/2}(I-e^{-t_iL_k})(u)(x)=\int_{\mathbb{R}^d}\mathcal{K}_i(x,y)u(y)d\mu_k(y)$$
for all   $u\in L^2(\mathbb{R}^2,d\mu_k)$ with  compact support
and for all a.e. $x\in \mathbb{R}^d$ such that $G.x\cap supp(u)= \emptyset$.
 \par Now one can write
 \begin{eqnarray*}
 &&\int_{ \mathbb{R}^d\setminus B_i^*} |R_j(I-e^{-t_iL_k})(b_i)(x)|d\mu_k(x)|
 \\&&\qquad\qquad\qquad\qquad\qquad\leq\int_{ \mathbb{R}^d\setminus B_i^*}  \int _{\mathbb{R}^d} |\mathcal{K}_i(x,y)||b_i(y)| d\mu_k(y)d\mu_k(x)
 \\ &&\qquad\qquad\qquad\qquad\qquad\leq   \int _{\mathbb{R}^d}\left\{\int_{ |x^+-y^+|>\sqrt{t_i}}   |\mathcal{K}_i(x,y)| d\mu_k(x)\right\}|b_i(y)|d \mu_k(y)
\end{eqnarray*}
 The second inequality follows from the fact that  for  $x\notin B_i^*$ and $y\in B_i$ we have for all $g\in G$
 $$|g.x-y|\geq |g(x)-x_i|-|y-x_i|>2\sqrt{t_i}-\sqrt{t_i}=\sqrt{t_i}.$$
and so, $|x^+-y^+|=\min_{g\in G}|g.x-y|>\sqrt{t_i}.$
\par We claim that there is a constant $C>0$  so that for all $i$
  $$\int_{ |x^+-y^+|>\sqrt{t_i}}  |\mathcal{K}_i(x,y)| d\mu_k(x)\leq C$$
 Indeed, by using the estimate (\ref{02})
 \begin{eqnarray*}
 && \int_{ |x^+-y^+|>\sqrt{t_i}}   |\mathcal{K}_i(x,y)| d\mu_k(x)
 \\&&\qquad\qquad\qquad\qquad\leq \int_{0}^{+\infty}g_i(s) \int_{ |x^+-y^+|>\sqrt{t_i}} |T_jW_s(x,y)|d\mu_k(x)ds
 \\&&\qquad\qquad\qquad\qquad\leq  C \int_{0}^{+\infty}\frac{ g_i(s)}{\sqrt{s}} e^{-c\sqrt{t_i/s}}
   \\&&\qquad\qquad\qquad\qquad\leq \int_{0}^{t_i} \frac{e^{-c\sqrt{t_i/s}} }{s} ds+\int_{t_i}^{+\infty}
\left(\frac{1}{\sqrt{s-t_i}}- \frac{1 }{\sqrt{s}}\right)e^{-c\sqrt{t_i/s}}s^{-1/2} ds
  \\&&\qquad\qquad\qquad\qquad \leq I_1+I_2.
 \end{eqnarray*}
We have
$$I_1=\int_0^1 \frac{e^{-c/\sqrt{u}} }{u} du.$$
and
$$I_2 \leq \int_{t}^{+\infty}
\left(\frac{1}{\sqrt{s-t}}- \frac{1 }{\sqrt{s}}\right) s^{-1/2} ds= \int_{0}^{+\infty}\left(\frac{1}{\sqrt{u(u+1)}}- \frac{1 }{u+1} \right)du  $$
  So, it follows that
 \begin{eqnarray*}
 \sum_i\int_{ \mathbb{R}^d\setminus B_i^*} |(R_j(I-e^{-t_iL_k})(b_i)(x)|d\mu_k(x)
 \leq C \sum_i\|b_i(y)\|_{1,k} \leq C\|f\|_{1,k}
\end{eqnarray*}
Thus in view of (\ref{012})  we deduce   that
\begin{eqnarray*}
  \mu_k\left\{x\notin\bigcup_iB_i^*;\;  R_j\left(\sum_i(I-e^{-t_iL_k})b_i\right)(x)>\lambda/4\right\}  \leq C\;  \frac{ \|f\|_{1,k}}{\lambda}.
  \end{eqnarray*}
  This concludes the proof  of  Theorem \ref{999}.
\end{proof}
We  turn now to   proving  Theorem \ref{101}. In order to  adapt   the arguments used in \cite[lemma 2.1 ]{JD2}  we shall first  prove some preliminaries lemmas \begin{lem}
 For all $x,y,y_0\in \mathbb{R}^d$ we have
 \begin{equation}\label{phi}
 \phi(x,y_0)\leq \phi(y,y_0)  e^{|x^+-y^+|}
\end{equation}
   \end{lem}
  \begin{proof}
  As  $\phi(.,y)$ is $G-invariant$ we have
  $ \phi(x,y)=  \phi(x^+,y)$. Then (\ref{phi})  is a consequence the following inequality
  \begin{eqnarray*}
   \sqrt{1+ |x^+-\eta|^2+|y_0|^2-|\eta|^2}&\leq& \sqrt{1+(|x^+-y^+|+|y^+-\eta|)^2+|y_0|^2-|\eta|^2}
  \\ &\leq&  |x^+-y^+|+\sqrt{1+ |y^+-\eta|^2+|y_0|^2-|\eta|^2},
\end{eqnarray*}
 for all $x,y,y_0\in\mathbb{R}^d$ and $\eta \in Conv(G.y_0)$
  \end{proof}
   \begin{lem}\label{456}
  Let $\lambda >0$ and  $\phi_\lambda$ be the function  given by
  $$\phi_\lambda(x,y)=\left\{\int_{\mathbb{R}^d}e^{ \sqrt{1+A(x,y,\eta)^2}}\;d\nu_y^G(\eta)\right\}^\lambda.$$
  The semigroup $W_t$ acting on the Hilbert space $L^2(\phi_{\lambda}(x,y)d\mu_k(x))$  with norms satisfying
  $$\|W_t\|_{L^2(\phi_\lambda(x,y)d\mu_k)\rightarrow L^2(\phi_\lambda(x,y)d\mu_k) }\leq Ce^{c \lambda^2 t}.$$
  Constants $C,c$ are independent of $\lambda$.
 \end{lem}
\begin{proof}
From the fact that  $W_t(x,y)\leq K_t(x,y)$ and $\|K_t(x,.)\|_{1,k}=1$ it follows that
 $$|W_t(u)(x)|^2\leq \int_{\mathbb{R}^d}K_t(x,z)|u(z)|^2w_k(z)dz$$
 and  from (\ref{phi})
  $$\int_{\mathbb{R}^d}|W_t(u)(x)|^2\phi_\lambda(x,y)d\mu_k(x) \leq\int_{\mathbb{R}^d}|u(z)|^2\phi_\lambda(z,y)\left(\int_{\mathbb{R}^d}K_t(x,z)e^{\lambda|x^+-z^+|}d\mu_k(x)\right)d\mu_k(z). $$
  Now   using  the  Gaussian upper bounds for the heat kernel (\ref{dhhk}) we get
  \begin{eqnarray*}
 \int_{\mathbb{R}^d}K_t(x,z)e^{\lambda|x^+-z^+|}d\mu_k(x) &\leq& C\; t^{-d/2}\int_{\mathbb{R}^d}e^{-|x^+-z^+|^2/4t}e^{\lambda|x^+-z^+|}dx
\\&\leq& C\; \sum_{g\in G}\int_{\mathbb{R}^d}e^{-|x-gz|^2/4t}e^{|x-gz|}dx\leq Ce^{c\lambda ^2t}.
  \end{eqnarray*}
  \end{proof}
\begin{cor}\label{666}
   Let $y\in \mathbb{R}^d$.  The semigroup $W_t=e^{-tL_k}$, $t>0$,  acting on $L^2( \phi(x,y)d\mu_k(x))$ has
the unique extension to a holomorphic semigroup $W_\zeta$, $\zeta\in \{\zeta\in \mathbb{C},\; |Arg\zeta|\leq \pi/4\}$
such that
 $$\|W_\zeta\|_{L^2(\phi(x,y)d\mu_k)\rightarrow L^2(\phi(x,y)d\mu_k) }\leq Ce^{c \Re \zeta},$$
 where $c$ and $C$ are independent of $y$.
 \end{cor}
    This can be  obtained  using  Lemma \ref{456} and a similar proof of  \cite[Prop. 3.2]{JD}.
 \begin{lem}\label{123}
 There exists a constant $C>0$ such that for all $x,y\in \mathbb{R}^d$ and $\alpha \in R^+$
   $$|T_j^2\phi(x,y)| +\left|\dfrac{T_j\phi(x,y)-T_j\phi(\sigma_\alpha.x,y)}{\langle x,\alpha\rangle}\right|\leq C \phi(x,y),
   \qquad j=1,...,d.$$
 \end{lem}
 \begin{proof}
   Let us begin with proving that
   \begin{equation}\label{Tj}
   \left| \frac{T_j\phi(x,y)-T_j\phi(\sigma_\alpha.x,y)}{\langle x,\alpha\rangle}\right| \leq C \phi(x,y).\end{equation}
 Noting  here  that
 $$T_j\phi(x,y)=\frac{\partial \phi}{\partial x_j}(x,y), $$
 since  $\phi$ is $G$-invariant.  We  have
 \begin{eqnarray*}
\frac{\partial \phi}{\partial x_j}(x,y)&=&
  \int_{\mathbb{R}^d}\frac{x_j-\eta_j}{\sqrt{1+A(x,y,\eta)^2} }\;e^{\sqrt{1+A(x,y,\eta)^2}}d\nu_{y}^G(\eta)
 \end{eqnarray*}
  and
   \begin{eqnarray*}
 \frac{\partial \phi}{\partial x_j}(\sigma_\alpha.x,y)
&=&\int_{\mathbb{R}^d}\frac{\sigma_\alpha.x_j- \eta_j}{\sqrt{1+A(\sigma_\alpha.x,y, \eta)^2} }\;e^{\sqrt{1+A(\sigma_\alpha.x,y, \eta)^2}}d\nu_{y}^G(\eta)
\\ &=& \int_{\mathbb{R}^d}\frac{\sigma_\alpha.x_j- \sigma_\alpha.\eta_j}{\sqrt{1+A( x,y, \eta)^2} }\;e^{\sqrt{1+A( x,y, \eta)^2}}
d\nu_{y}^G(\eta).
\end{eqnarray*}
Using (\ref{sig}) we get  that
\begin{eqnarray*}
\dfrac{T_j\phi(x,y)-T_j\phi(\sigma_\alpha.x,y)}{\langle x,\alpha\rangle}
 = \frac{  \alpha_j}{\langle x,\alpha\rangle}
 \int_{\mathbb{R}^d} \langle x-\eta,\alpha\rangle\;\frac{e^{\sqrt{1+A(x,y,\eta)^2}} }{ \sqrt{1+A(x,y,\eta)^2} }\;d\nu_{y}^G(\eta).
\end{eqnarray*}
 As the measure $\nu_{y}^G $ is $G$- invariant,
$$
\int_{\mathbb{R}^d}\langle x-\eta,\alpha\rangle\; \frac{e^{\sqrt{1+A(x,y,\eta)^2}} }{ \sqrt{1+A(x,y,\eta)^2} }\;d\nu_{y}^G(\eta)=-
\int_{\mathbb{R}^d} \langle \sigma_\alpha.x-\eta,\alpha\rangle\;\frac{e^{\sqrt{1+A(\sigma_\alpha.x,y,\eta)^2}} }{ \sqrt{1+A(\sigma_\alpha.x,y,\eta)^2} }\;d\nu_{y}^G(\eta)
$$
and so we can write
\begin{eqnarray*}
&&\int_{\mathbb{R}^d} \langle x-\eta,\alpha\rangle\;\frac{e^{\sqrt{1+A(x,y,\eta)^2}} }{ \sqrt{1+A(x,y,\eta)^2} }\;d\nu_{y}^G(\eta)
\\&=& \frac{1}{2}\int_{\mathbb{R}^d} \langle x-\eta,\alpha\rangle\;\left\{\frac{e^{\sqrt{1+A(x,y,\eta)^2}} }{ \sqrt{1+A(x,y,\eta)^2} }
-\frac{e^{\sqrt{1+A(\sigma_\alpha.x,y,\eta)^2}} }{ \sqrt{1+A(\sigma_\alpha.x,y,\eta)^2} }\right\}\;d\nu_{y}^G(\eta)
\\&&+\frac{1}{2} \langle x, \alpha\rangle |\alpha|^2\;\int_{\mathbb{R}^d}  \;\frac{e^{\sqrt{1+A(x,y,\eta)^2}} }{ \sqrt{1+A(x,y,\eta)^2} }\;d\nu_{y}^G(\eta)
\\&=& I_1(x,y)+I_2(x,y)
\end{eqnarray*}
Clearly
\begin{equation}\label{I1}
  |I_2(x,y)|\leq C|\langle x, \alpha\rangle|\phi(x,y).
\end{equation}
To obtain the same bound for $I_1(x,y)$ we proceed as the following: applying the mean value Theorem on the function
$s\rightarrow e^s/s$ for $s>1$ which is increasing function,  we get  that
\begin{eqnarray*}
&& \left|\frac{e^{\sqrt{1+A(x,y,\eta)^2}} }{ \sqrt{1+A(x,y,\eta)^2} }
-\frac{e^{\sqrt{1+A(\sigma_\alpha.x,y,\eta)^2}} }{ \sqrt{1+A(\sigma_\alpha.x,y,\eta)^2} }\right|
\\&&\qquad\qquad\qquad\qquad\leq |\sqrt{1+A(x,y,\eta)^2}-\sqrt{1+A(\sigma_\alpha.x,y,\eta)^2}|
\\&&\qquad\qquad\qquad\qquad\qquad\qquad\times
   \frac{e^{ \max (\;\sqrt{1+A(x,y,\eta)^2},\sqrt{1+A(\sigma_\alpha.x,y,\eta)^2}\;)} }
   {\max (\sqrt{1+A(x,y,\eta)^2},\sqrt{1+A(\sigma_\alpha.x,y,\eta)^2})}
    \end{eqnarray*}
 We  have
 \begin{eqnarray*}
   |\sqrt{1+A(x,y,\eta)^2}-\sqrt{1+A(\sigma_\alpha.x,y,\eta)^2}|=
   \frac{||x-\eta|^2-|\sigma_\alpha x-\eta|^2|}{\sqrt{1+A(x,y,\eta)^2}+\sqrt{1+A(\sigma_\alpha.x,y,\eta)^2}}
 \end{eqnarray*}
As
\begin{eqnarray*}
  ||x-\eta|^2-|\sigma_\alpha x-\eta|^2|&=&|\langle x-\sigma_\alpha.x, x+\sigma_\alpha.x-2\eta\rangle|
\\& =& |\langle x,\alpha \rangle||\langle  \alpha, x- \eta+ \sigma_\alpha.x-\eta\rangle|
\\& \leq&  C |\langle x,\alpha \rangle|\Big(| x- \eta|+|\sigma_\alpha.x-\eta| \Big)
\end{eqnarray*}
and
$$| x- \eta|+|\sigma_\alpha.x-\eta|\leq \Big(\sqrt{1+A(x,y,\eta)^2}+\sqrt{1+A(\sigma_\alpha.x,y,\eta)^2}\Big)$$

It follows that
\begin{equation}\label{111}
  |\sqrt{1+A(x,y,\eta)^2}-\sqrt{1+A(\sigma_\alpha.x,y,\eta)^2}|\leq C |\langle x,\alpha\rangle|
\end{equation}
However, since
   $$|x-\eta |\leq \max (\sqrt{1+A(x,y,\eta)^2},\sqrt{1+A(\sigma_\alpha.x,y,\eta)^2}) $$
 hence we obtain
\begin{equation}\label{222}
 |x-\eta | \left|\frac{e^{\sqrt{1+A(x,y,\eta)^2}} }{ \sqrt{1+A(x,y,\eta)^2} }
-\frac{e^{\sqrt{1+A(\sigma_\alpha.x,y,\eta)^2}} }{ \sqrt{1+A(\sigma_\alpha.x,y,\eta)^2} }\right|\leq
C\Big(e^{\sqrt{1+A(x,y,\eta)^2}}+e^{\sqrt{1+A(\sigma_\alpha.x,y,\eta)^2}}\Big)
\end{equation}
Therefore combine (\ref{111}) with (\ref{222}) and using the $G$-invariance of the measure $d\nu_y^G$ we obtain that
$$| I_1(x,y)|\leq C \phi(x,y)$$
\par We come now to the estimate of  $T_j^2 \phi (x,y)$. We  write
\begin{eqnarray*}
T_j^2 \phi (x,y)=\qquad\qquad\qquad\qquad\qquad\qquad\qquad\qquad\qquad\qquad
\end{eqnarray*}
\begin{eqnarray*}
&& \int_{\mathbb{R}^d}\frac{(x_j-\eta_j)^2}{1+A(x,y,\eta)^2}\; e^{\sqrt{1+A(x,y,\eta)^2}}d\nu^G_y+
\int_{\mathbb{R}^d}\frac{1}{\sqrt{1+A(x,y,\eta)^2}}\; e^{\sqrt{1+A(x,y,\eta)^2}}d\nu^G_y
\\&&\qquad\qquad\qquad-\int_{\mathbb{R}^d}\frac{(x_j-\eta_j)^2}{(1+A(x,y,\eta)^2)^{3/2}}e^{\sqrt{1+A(x,y,\eta)^2}}d\nu^G_y
\\&&\qquad\qquad\qquad\qquad\qquad\qquad\qquad+\sum_{\alpha\in R^+}k(\alpha)\alpha_j\left\{
\dfrac{T_j\phi(x,y)-T_j\phi(\sigma_\alpha.x,y)}{\langle x,\alpha\rangle}\right\}d\nu^G_y.
\end{eqnarray*}
 Clearly each integral can be  estimate  by $C \phi (x,y)$. This finishes the proof of the lemma.
\end{proof}

  \begin{lem}\label{555}
   There exists $C>0$ such that for $y\in \mathbb{R}^d$ and real valued  function  $f\in L^2(\mathbb{R}^2,\phi(x,y)d\mu_k(x))$  with
   $T_jf\in  L^2(\mathbb{R}^2, d\mu_k(x))$   we have
   \begin{equation}\label{124}
  \left| \int_{\mathbb{R}^d} T_jf(x) f(x)T_j\phi(x,y) d\mu_k(x) \right|\leq C \int_{\mathbb{R}^d}|f(x)|^2 \phi(x,y) d\mu_k(x).
\end{equation}
     \end{lem}
  \begin{proof}
  From (\ref{DO}) it follows that
  \begin{eqnarray*}
      && 2\int_{\mathbb{R}^d} T_jf(x) f(x) T_j\phi(x,y) d\mu_k(x) = -\int_{\mathbb{R}^d}  f^2(x)T_j^2\phi(x,y) d\mu_k(x)\\
       && \qquad\qquad\qquad\qquad\qquad- \sum_{\alpha\in R^+}k(\alpha)\alpha_j\int_{\mathbb{R}^d} \frac{\Big(f(x)-f(\sigma_\alpha.x)\Big)^2}{\langle x,\alpha\rangle}T_j\phi(x,y) d\mu_k(x).
 \end{eqnarray*}
We write  the second integral as
 \begin{eqnarray*}
   &&\int_{\mathbb{R}^d} \frac{\Big(f(x)-f(\sigma_\alpha.x)\Big)^2}{\langle x,\alpha\rangle}T_j\phi(x,y) d\mu_k
 \\ &&= \int_{\mathbb{R}^d}  \Big(f(x)-f(\sigma_\alpha.x)\Big)^2  \frac{\Big(T_j\phi(x,y)-T_j\phi(\sigma_\alpha.x,y)\Big)}{2\langle x,\alpha\rangle} d\mu_k
 \end{eqnarray*}
Thus  (\ref{124}) follows  from Lemma \ref{123}.
 \end{proof}
\par We come now to the proof of Theorem \ref{101}
\begin{proof}[Proof of Theorem \ref{101}]
\par  Fixing $y \in \mathbb{R}^d$.  Let $Q$ be the quadratic form associated to the infinitesimal generator of $W_t$ considered on the Hilbert $L^2( \phi(x,y)d\mu_k(x))$.  We have
  \begin{eqnarray*}
  Q(f,g)&=& \sum_{j=1}^d \int_{\mathbb{R}^d}T_jf(x)\overline{T_jg(x)} \phi(x,y) d\mu_k(x) + \int_{\mathbb{R}^d} T_jf(x) \overline{g(x)}T_j\phi(x,y) d\mu_k(x)
  \\&&\qquad\qquad\qquad\qquad+ \int_{\mathbb{R}^d}f(x)\overline{g(x)}\phi(x,y)V(x)d\mu_k(x)
  \end{eqnarray*}
  $$D(Q)=\{f;\; f, V^{1/2}f, T_jf\in L^2(\mathbb{R}^d, \phi(x,y)d\mu_k(x)),\; j=1,...,d\}$$
Define
$$\|f\|_Q^2= \int_{\mathbb{R}^d}\left\{\sum_{j=1}^d |T_jf(x)|^2+|f(x)|^2V(x)+|f(x)|^2\right\}\phi(x,y)d\mu_k(x) $$
 Using Lemma \ref{555} we have that
 $$  |Q(f,f)|\leq  C_1 \|f\|_Q^2, \quad \text{and }\quad \|f\|_Q^2\leq C_2|Q(f,f)|+C_3\|f\|^2_{L^2( \phi(x,y)d\mu_k)}$$
 with $C_1, C_2$ and $C_3$ independent of $y$ and $V$.
 According to the corollary  (\ref{666})  and  from the holomorphic semigroups theory   (see e.g \cite[th 3.1]{Ro} we have
 \begin{equation}\label{777}
   \|L_kW_t(g)\|_{L^2(\phi(x,y)d\mu_k(x))}\leq ct^{-1}e^{ct}\|g\|_{L^2(\phi(x,y)d\mu_k(x))}.
\end{equation}
 Setting $ g(x) = W_{1/2}(x, y)$  and $f (x) =W_{1/2}(g)(x) = W_1(x, y) $.  We get from the boundedness of $W_t$ and (\ref{777})
 \begin{eqnarray*}
   &&\| T_j(W_1(x,y))\|_{L^2(\phi(x,y)d\mu_k(x))}\qquad\qquad\qquad\qquad\qquad\qquad\qquad\qquad\qquad\qquad\qquad
   \\&& \qquad\qquad\qquad\leq  \|f\|_Q^2\leq   C_2|Q(f,f)|+C_3\|f\|^2_{L^2( \phi(x,y)d\mu_k(x))}
   \\&&\qquad\qquad\qquad\leq   C\|  L_k f\|_{L^2(\phi(x,y)d\mu_k)} \|f\|_{L^2( \phi(x,y)d\mu_k)}+ C_3\|f\|^2_{L^2( \phi(x,y)d\mu_k)}
   \\&& \qquad\qquad\qquad\leq C'\|g\|^2_{L^2( \phi(x,y)d\mu_k(x))}\leq C'',
 \end{eqnarray*}
 with $C''$ independent of $y$ and $V$. Now   the desired estimate (\ref{01}) follows  from the fact that    $W_t(x,y)=t^{-d/2-\gamma_k} \tilde{W}_1(x/\sqrt{t},y\sqrt{t})$ where
$\tilde{W}_s$  is the semigroup generated by  $A_k-tV(\sqrt{t}x)$ ( this can be seen using Trotter formula).
\par To proving (\ref{02}) we proceed as follows. Observe that for all $x,y\in \mathbb{R}^d$
\begin{eqnarray*}
 1&=&\left\{\int_{\mathbb{R}^d} e^{\frac{1}{4}\sqrt{1+A(x,y,\eta)^2}} e^{-\frac{1}{4}\sqrt{1+A(x,y,\eta)^2}}d\nu_y^G(\eta)\right\}^2
 \\&\leq&
 \int_{\mathbb{R}^d} e^{\frac{1}{2} \sqrt{1+A(x,y,\eta)^2}}  d\nu_y^G(\eta)\int_{\mathbb{R}^d}  e^{ -\frac{1}{2}\sqrt{1+A(x,y,\eta)^2}}  d\nu_y^G(\eta)
 \end{eqnarray*}
Then  Cauchy-Schwartz inequality and (\ref{nu}) yield
\begin{eqnarray*}
&&\int_{ |x^+-y^+|>\sqrt{t}}|T_jW_s(x,y)| d\mu_k(x)\leq \left\{\int_{\mathbb{R}^d}|T_jW_t(x,y)|^2\phi(x/\sqrt{s},y/\sqrt{s})d\mu_k(x)\right\}^{1/2}
\\ &&\times \left\{\int_{ |x^+-y^+|>\sqrt{t}} \left(\int_{\mathbb{R}^d}  e^{ -\sqrt{1+A(x ,y ,\eta)^2/s}} \; d\nu_{y}^G(\eta)\right) d\mu_k(x)\right\}^{1/2}
\end{eqnarray*}
   We have
   \begin{eqnarray*}
   &&\int_{ |x^+-y^+|>\sqrt{t}} \left(\int_{\mathbb{R}^d}  e^{ -\sqrt{1+A(x ,y ,\eta)^2/s}}  d\nu_{y}^G(\eta)\right) d\mu_k(x)
  \\ &&\qquad\qquad\qquad\qquad\leq \int_{ |x^+-y^+|>\sqrt{t}} \left(\int_{\mathbb{R}^d}  e^{ - A(x ,y ,\eta)/\sqrt{s}}\;  d\nu_{y}^G(\eta)\right) d\mu_k(x)
  \\&&\qquad\qquad\qquad\qquad \leq \int_{ \mathbb{R}^d} \left(\int_{\mathbb{R}^d}  e^{ - A(x ,y ,\eta)/\sqrt{s}} \mathds{1}_{A(x,y,\eta)>\sqrt{t}}  \; d\nu_{y}^G(\eta)\right) d\mu_k(x)
  \\&&\qquad\qquad\qquad\qquad= \leq \int_{ \mathbb{R}^d}  \tau_{-y}\Big( e^{-|.|/\sqrt{s}} \mathds{1}_{|.|>\sqrt{t}}\Big)(x)   d\mu_k(x)
  \\ &&\qquad\qquad\qquad\qquad= \int_{ |x|>\sqrt{t}}   e^{-|x|/\sqrt{s}}     d\mu_k(x)\leq s^{\delta_k+d}e^{-\frac{1}{2}\sqrt{t/s}}.
 \end{eqnarray*}
 This concludes  (\ref{02}).
    \end{proof}
\section{$L^p$ smoothing for the Dunk- Schr\"{o}dinger operators semigroup}
We recall the definition of the classical Kato class of potential $V$, see \cite{Sim1}
 \begin{defn}

 A measurable function  $V$  on $\mathbb{R}^d$ belongs to the Kato class $ \mathbb{K}_d$ if
      \begin{eqnarray*}
  && \lim_{t\downarrow 0}\left[\sup_{x}\int_{ |x -y|\leq t}  |x-y|^{2-d} |V(y)| dy\right]=0 \quad\text{when}\quad d\geq 3
\\&& \lim_{t\downarrow 0}\left[\sup_{x}\int_{ |x-y|\leq t} \ln\{ |x-y|^{-1}\}  |V(y)| dy\right]=0 \quad\text{when}\quad d=2
\\&&\sup_{x}\int_{ |x-y|\leq1}   |V(y)| dy< \infty \quad\text{when}\quad d=1
      \end{eqnarray*}
              \end{defn}

\begin{prop}\label{prop1}

      A measurable function  $V$  on $\mathbb{R}^d$ belongs to the Kato class $ \mathbb{K}_d$ if and only if
      \begin{eqnarray*}
  && \lim_{t\downarrow 0}\left[\sup_{x}\int_{ |x^+-y^+|\leq t}  |x^+-y^+|^{2-d} |V(y)| dy\right]=0 \quad\text{when}\quad d\geq 3
\\&& \lim_{t\downarrow 0}\left[\sup_{x}\int_{ |x^+-y^+|\leq t} \ln\{ |x^+-y^+|^{-1}\}  |V(y)| dy\right]=0 \quad\text{when}\quad d=2
\\&&\sup_{x}\int_{ |x^+-y^+|\leq1}   |V(y)| dy< \infty \quad\text{when}\quad d=1.
      \end{eqnarray*}
              \end{prop}
\begin{proof}
   Remember here that
\begin{equation}\label{12}
  |x^+-y^+|= \min_{g\in G}|g.x-y|
\end{equation}
and
\begin{equation}\label{gb}
  \bigcup_{g\in G}B(g.x,t)=\{y\in \mathbb{R}^d,\; |x^+-y^+|\leq t \;\}.
  \end{equation}
where  $B(x,r)=\{x\in \mathbb{R}^d;\; |x-y|\leq r\}$    the closed ball of centre $x$ and radius  $r>0$.
Clearly we have
\begin{eqnarray*}
   \int_{ |x-y|\leq t}  |x-y|^{2-d} |V(y)| dy&\leq& \int_{ |x^+-y^+|\leq t}  |x^+-y^+|^{2-d} |V(y)| dy  \quad\text{for}\; d\geq 3
\\ \int_{ |x-y|\leq t}   \ln\{|x-y|^{-1}\} |V(y)| dy&\leq& \int_{ |x^+-y^+|\leq t}  \ln\{|x^+-y^+|^{-1}\}  |V(y)| dy  \quad\text{for}\; d=2
\\\int_{ |x-y|\leq t}   |V(y)| dy&\leq& \int_{ |x^+-y^+|\leq t}    |V(y)| dy  \quad\text{for}\; d=1,
\end{eqnarray*}
which implies the if part.  Using (\ref{gb}) we have for $d\geq3$
\begin{eqnarray*}
   \int_{ |x^+-y^+|\leq t}  |x^+-y^+|^{2-d} |V(y)| dy&=& \sum_{g\in G} \int_{\{ y\in g(C); |x^+-y^+|\leq t \}}  |x^+-y^+|^{2-d} |V(y)| dy
   \\ &=&\sum_{g\in G} \int_{\{ y\in g(C); |x^+-g^{-1}(y)|\leq t \}}  |x^+-g^{-1}(y)|^{2-d} |V(y)| dy
   \\ &=&\sum_{g\in G} \int_{ \{y\in g(C); |g(x^+)-y|\leq t\} }  |g(x^+)-y|^{2-d} |V(y)| dy
   \\&\leq & \sum_{g\in G} \int_{ |g(x^+)-y|\leq t }  |g(x^+)-y|^{2-d} |V(y)| dy.
\end{eqnarray*}
Similarly we have
\begin{eqnarray*}
  \int_{ |x^+-y^+|\leq t}   \ln\{|x^+-y^+|^{-1}\} |V(y)| dy&\leq& \sum_{g\in G} \int_{ |g(x^+)-y|\leq t } \ln\{|g(x^+)-y|^{-1}\}  |V(y)| dy  \quad\text{for}\; d=2
\\ \int_{ |x^+-y^+|\leq t}    |V(y)| dy&\leq& \sum_{g\in G} \int_{ |g(x^+)-y|\leq t } |V(y)| dy  \quad\text{for}\; d=1.
\end{eqnarray*}
This achieve  the converse  part and conclude the proposition \ref{prop1}.
\end{proof}
\begin{cor}\label{1}
If $V\in \mathbb{K}_d$ then
  $$\sup_{x\in \mathbb{R}^d} \int_{|x^+-y^+|\leq 1}| V(y)|dy< \infty.$$
\end{cor}
 This can be seen  from \cite[ Lemma 4.4]{Sim1} and the fact that
$$  \int_{|x^+-y^+|\leq 1}| V(y)|dy\leq \sum_{g\in G}   \int_{|y-gx|\leq 1}| V(y)|dy.$$
In what follows we need  the following  lemma.
\begin{lem}
There exists a constant $c>0$ such that for all $x\in \mathbb{R}^d$ and $r>0$  the ball $B(x,r)$ can be covered  by $N$ balls all of radius $1$  where
  $N\leq c(r+1)^d$.
\end{lem}\label{ball}
\begin{proof}
 Observe that $|x-y|\leq r$ implies that for all $i=1,2,.,.,.,d$
$$|x_i-y_i|\leq r\leq  \frac{2([rd/2]+1)}{d}.$$
  Let $x=(x_1,x_2,.,.,x_d)$. Putting  for  $ j=1,.,.,., 2([rd/2]+1)$
  $$a_i^j=x_i-\frac{2([rd/2]+1)}{d}+\frac{2j-1}{d}.$$
   Clearly if $|x_i-y_i|\leq r$ then there exists  $j_i$ such that $|a_i^{j_i}-y_i|\leq  1/d$
 and if we let
  $a^{j_1,.,.,.,j_d}= (a_1^{j_1},.,.,a_d^{j_d})$ then we have
  $$|a^{j_1,.,.,.,j_d}-y|\leq d\max_i|a_i^{j_i}-y_i|\leq 1.$$
  The number of the balls of center $a^{j_1,.,.,.,j_d}$ and of radius 1 is  $2^d([rd/2]+1)^d$ which is less than  $(2d)^d(r+1)^d$.
\end{proof}
As a  consequence we have the following result
\begin{cor}\label{15} If $V\in \mathbb{K}_d$ then there exists $C>0$ such that for $r>0$
  $$ \int_{ |x^+-y^+|\leq  r}|V(y)| dy \leq  C (r+1)^d.$$
\end{cor}
 \par The Kato class can be characterized by means of the heat kernel $K_t$,
\begin{thm}\label{th1}
      If $ V\in \mathbb{K}_d$  then
     $$\lim_{t\downarrow 0}\left\|\int_0^te^{-sA_k}|V|ds\right\|_\infty= \lim_{t\downarrow 0}\left[\sup_{x}\int_0^t\int_{ \mathbb{R}^d} K_s(x,y) |V(y)| d\mu_k(y)ds \right]=0.$$
\end{thm}
\begin{proof}
   Put
  $$Q(x,y,t)=\int_0^t K_s(x,y)ds$$
  For $d\geq 3$, in view of  (\ref{dhk}), we obtain that
  \begin{equation}\label{in3}
    Q(x,y,t)\leq \int_0^{+\infty } K_s(x,y)ds\leq C\;  \frac{|x^+-y^+|^{2-d}}{w_k(y)}.
  \end{equation}
   In addition for $d\geq 1$ and as  the function
  $$s\rightarrow \frac{ 1}{s^{  d/2} }e^{\frac{-c|x^+-y^+|^2}{s}}  .$$
  is increasing for $s < \frac{2c}{d}\;|x^+-y^+|^2$, then when  $t < \frac{2c}{d}\;|x^+-y^+|^2$
  \begin{equation}\label{q2}
  Q(x,y,t)\leq  \frac{ C}{t^{  d/2-1} w_k(y)}e^{\frac{-c|x^+-y^+|^2}{t}}.
  \end{equation}
  Now  suppose that $d\geq 2$.  For  $0<t< c/d$  we let
   $   \beta= \left(\frac{dt}{2c} \right)^{1/2d}$. This clearly imply  that
   \begin{equation}\label{14}
    t< \frac{2c}{d}\;\beta^2.
   \end{equation}
   Write
  \begin{eqnarray*}
  \int_{\mathbb{R}^d}Q(x,y,t)|V(y)|d\mu_k(y)&\leq& \int_{ |x^+-y^+|>\beta}Q(x,y,t)|V(y)|d\mu_k(y)\\&&+ \int_{ |x^+-y^+|\leq \beta}Q(x,y,t)|V(y)|d\mu_k(y)
\end{eqnarray*}
In view of (\ref{14}) and (\ref{q2}), we have
\begin{eqnarray*}
  \int_{ |x^+-y^+|>\beta}Q(x,y,t)|V(y)|d\mu_k(y)&\leq& C\; t \int_{ |x^+-y^+|>\beta} t^{ - d/2}  e^{\frac{-c|x^+-y^+|^2}{t}}|V(y)| dy
   \\ &\leq& C t  \sum_{n\geq 1} \int_{n\beta \leq  |x^+-y^+|\leq (n+1)\beta}t^{ - d/2}  e^{\frac{-c| x^+-y^+|^2}{t}}|V(y)| dy
\\ &\leq& C t  \sum_{n\geq 1}t^{ - d/2}  e^{\frac{-cn^2\beta^2}{t}}\int_{ |x^+-y^+|\leq (n+1)\beta}|V(y)| dy
\end{eqnarray*}
 By using corollary  \ref{15} and (\ref{14})  we obtain
\begin{eqnarray*}
  \int_{ |x^+-y^+|>\beta}Q(x,y,t)|V(y)| \mu_k(y)&\leq& C\; t \sum_{n\geq 1}t^{ - d/2}  e^{\frac{-cn^2\beta^2}{2t}}e^{\frac{-cn^2\beta^2}{2t}}n^{d}
 \\&\leq& C\; \frac{t}{\beta^d} \sum_{n\geq 1}\left(\frac{n^{2}\beta^2}{t}\right)^{\frac{d}{2}}  e^{\frac{-cn^2\beta^2}{2t}}e^{\frac{-cn^2\alpha^2}{2t}}
\\&\leq& C\; \dfrac{t}{\beta^{d}} \sum_{n\geq 1}   e^{\frac{-cn^2\beta^2}{2t}}
\\&\leq& C\; \dfrac{t}{\beta^{d}} \sum_{n\geq 1}   e^{-dn^2/4}
\\&\leq& C\; \dfrac{t}{\beta^{d }} \leq C \sqrt{t}
\end{eqnarray*}
Therefore
$$\lim_{t\downarrow0}\sup_x\int_{ |x^+-y^+|>\beta}Q(x,y,t)|V(y)|d\mu_k(y)=0.$$
On the other hand, when  $ d\geq 3 $ we have by using (\ref{in3}) and  Proposition \ref{prop1}
 $$\lim_{t\downarrow0}\sup_x\int_{ |x^+-y^+| \leq\beta}Q(x,y,t)|V(y)|d\mu_k(y)=0.$$
 In the case  $d=2$
 \begin{eqnarray*}
  \int_{ |x^+-y^+|\leq \beta}Q(x,y,t)|V(y)|d\mu_k(y) &=& \int_{ |x^+-y^+|\leq \beta} \int_{0}^{t} e^{-s} e^{s} K_s(x,y) |V(y)|d\mu_k(y) ds
 \\&\leq& e^{t} \int_{ |x^+-y^+|\leq \beta} \int_{0}^{\infty} e^{-s} K_s(x,y) |V(y)|d\mu_k(y) ds
 \\&\leq & C e^{t} \int_{ |x^+-y^+|\leq \beta}  \int_{0}^{\infty}  \frac{  e^{-s}}{s }\; e^{\frac{-c|x^+-y^+|^2}{s}}ds |V(y)|dy.
   \end{eqnarray*}
   We write
   $$ \int_{0}^{\infty}  \frac{  e^{-s}}{s }\; e^{\frac{-c|x^+-y^+|^2}{s}}ds= \int_{0}^{1}  \frac{  e^{-s}}{s }\; e^{\frac{-c|x^+-y^+|^2}{s}}ds
    +\int_{1}^{\infty}  \frac{  e^{-s}}{s }\; e^{\frac{-c|x^+-y^+|^2}{s}}ds.$$
   For the second integral we have
    $$ \int_{1}^{\infty}  \frac{  e^{-s}}{s }\; e^{\frac{-c|x^+-y^+|^2}{s}}ds\leq  \int_{1}^{\infty}  \frac{  e^{-s}}{s }\;ds\leq - C\ln |x^+-y^+|, $$
    since $|x^+-y^+|\leq (1/2)^{1/2d}$.
   While for the first integral we write
   $$\int_{0}^{1}  \frac{  e^{-s}}{s }\; e^{\frac{-c|x^+-y^+|^2}{s}}ds=\int_{0}^{ |x^+-y^+|^2}  \frac{  e^{-s}}{s }\; e^{\frac{-c|x^+-y^+|^2}{s}}ds+\int_{|x^+-y^+|^2}^{1}  \frac{  e^{-s}}{s }\; e^{\frac{-c|x^+-y^+|^2}{s}}ds$$
   We have
   \begin{eqnarray*}
 \int_{0}^{ |x^+-y^+|^2}  \frac{  e^{-s}}{s }\; e^{\frac{-c|x^+-y^+|^2}{s}}ds&\leq& \int_{0}^{ |x^+-y^+|^2}  \frac{  1}{s }\; e^{\frac{-c|x^+-y^+|^2}{s}}ds=\int_{0}^{ 1}  \frac{ 1 }{s }\; e^{\frac{-c }{s}}ds
 \\&\leq& - C\ln |x^+-y^+|
  \end{eqnarray*}
    and
    \begin{eqnarray*}
     \int_{|x^+-y^+|^2}^{1}  \frac{  e^{-s}}{s }\; e^{\frac{-c|x^+-y^+|^2}{s}}ds&\leq& \int_{|x^+-y^+|^2}^{1}  \frac{  e^{-s}}{s }\;
   \\&\leq& \int_{0}^{1}  \frac{  e^{-s}-1}{s } ds+\int_{|x^+-y^+|}^{1}  \frac{  1}{s } ds \leq  - C\ln |x^+-y^+|
\end{eqnarray*}
      It then follows from the proposition  \ref{prop1} that
    $$\lim_{t\downarrow0}\sup_x\int_{ |x^+-y^+|\leq \beta}Q(x,y,t)|V(y)|d\mu_k(y)=0.$$
  \par For $ d= 1 $,
  \begin{eqnarray*}
       \int_{ |x^+-y^+| >1 }Q(x,y,t)|V(y)|d\mu_k(y) &\leq&  C \int_0^t \frac{ds}{\sqrt{s}} \int_{ |x^+-y^+|> 1}e^{ -c|x^+-y^+|^2 }  |V(y)|dy
       \\&\leq & C \sqrt{t}\sum_{n=1}^\infty ne^{-cn^2}\leq C\sqrt{t}
 \end{eqnarray*}

$$
  \int_{ |x^+-y^+|\leq 1}Q(x,y,t)|V(y)| \mu_k(y) \leq  C \int_0^t \frac{ds}{\sqrt{s}} \int_{ |x^+-y^+|\leq 1}  |V(y)|dy
  \leq  C \sqrt{t}$$
 This completes the proof.
\end{proof}
\begin{cor}
  If $V\in \mathbb{K}_d  $ then
  $$\lim_{a\rightarrow\infty}\|(A_k+a)^{-1}|V|\|_{\infty}=0$$
  \begin{proof}
For $t>0$, we write
\begin{eqnarray*}
  ((A_k+a)^{-1}|V|)(x)&=&\int_0^\infty e^{-sa}e^{-sA_k}(|V|)(x)ds \\
  &=& \sum_{n\geq 0} \int_{nt}^{(n+1)t}e^{-sa}e^{-sA_k}(|V|)(x)ds
  \\&=&\sum_{n\geq 0}e^{-nat} \int_{0}^{t}e^{-sa}e^{-(s+nt)A_k}(|V|)(x)ds.
 \end{eqnarray*}
  Using the integral representation of $ e^{-ntA_k}$, we obtain
  \begin{eqnarray*}
  ((A_k+a)^{-1}|V|)(x)
  &=&\sum_{n\geq 0}e^{-nat} \int_{\mathbb{R}^d}K_{nt}(x,y)\int_{0}^{t}e^{-sa}e^{- sA_k}(|V|)(y)ds dy
  \\&\leq &\sum_{n\geq 0}e^{-nat}  \left\|\int_{0}^{t} e^{- sA_k}(|V|)(y)ds \right\|_\infty\int_{\mathbb{R}^d}K_{nt}(x,y)dy
  \end{eqnarray*}
  Since
  $$\int_{\mathbb{R}^d}K_{nt}(x,y)dy=1$$
  Thus we get
 $$\|(A_k+a)^{-1}|V|\|_\infty \leq \frac{1}{1-e^{-at}}\left\|\int_{0}^{t} e^{- sA_k}(|V|) ds \right\|_\infty$$
  and the corollary follows from Theorem \ref{th1},  by  letting $t=1/a \rightarrow0$.
  \end{proof}
\end{cor}
  Similarly to the usual case ( see e.g. \cite[Prop. 3.35]{LH}, one can  deduce  that the operator   $V$ (  as a multiplication operator )is $A_k$-form bounded with relative form bound $0$.
Then from the   the well known K.L.M.N. theorem, (see \cite[Th X.17]{Sim2}) the  Schrodinger operators  $ A_k+V$ is well defined and self-adjoint as a form sum.
\begin{thm}
 If $V\in  \mathbb{K}_d$, and $t > 0$, then $W_t=e^{-tL_k}$ is a bounded operator
  from $L^{p}(\mathbb{R}^d,d\mu_k)$ to  $L^{q}(\mathbb{R}^d,d\mu_k)$ for all $1\leq p\leq q\leq \infty.$
  \end{thm}
 It appears that the proof of this theorem is almost the same as the proof of the Theorem 2.1 in  \cite{Sim1}.  It consists,
 by making use of the Riesz-Thorin theorem,   to
prove that $W_t$ is bounded from $L^\infty(\mathbb{R}^d,d\mu_k)$ to  $L^\infty(\mathbb{R}^d,d\mu_k)$,   $L^1(\mathbb{R}^d,d\mu_k)$ to  $L^1(\mathbb{R}^d,d\mu_k)$ and   $L^1(\mathbb{R}^d,d\mu_k)$ to  $L^\infty(\mathbb{R}^d,d\mu_k)$.  The details will be omitted.

  \end{document}